\documentclass[12pt,a4paper]{article}%
\usepackage{amsmath}
\usepackage{amssymb}
\usepackage{makeidx}
\usepackage{amsfonts}
\usepackage{tikz}%
\DeclareMathOperator{\diam}{diam}
\DeclareMathOperator{\R}{Range}
\newtheorem{theorem}{Theorem}[section]

\newtheorem{claim}[theorem]{Claim}

\newtheorem{definition}[theorem]{Definition}
\newtheorem{example}[theorem]{Example}
\newtheorem{lemma}[theorem]{Lemma}
\newtheorem{proposition}[theorem]{Proposition}
\newtheorem{remark}[theorem]{Remark}

\newenvironment{proof}[1][Proof]{\noindent\textbf{#1.} }{\rule{0.5em}{0.5em}}
\newenvironment{acknowledgement}[1][Acknowledgement]{\textbf{#1.} }{ }

\begin{document}

\title{On the spectrum of the hierarchical Laplacian}
\author{Alexander Bendikov\thanks{Research supported by the Polish Government Scientific Research Fund, Grant 2012/05/B/ST 1/00613}\\
\small{bendikov@math.uni.wroc.pl}
\and
Pawe{\l } Krupski\\
\small{Pawel.Krupski@math.uni.wroc.pl}
\and
\normalsize{Mathematical Institute, University of Wroc\l aw}\\
\normalsize{Pl.Grunwaldzki 2/4, 50-384~Wroc\l aw, Poland}}

\maketitle

\begin{abstract}
Let $(X,d)$ be a locally compact separable ultrametric space. We assume that
$(X,d)$ is proper, that is, any closed ball $B\subset X$ is a compact set.
Given a measure $m$ on $X$ and a function $C(B)$ defined on the set of balls
(the choice function) we define the hierarchical Laplacian $L_{C}$ which is
closely related to the concept of the hierarchical lattice of F.J. Dyson.
$L_{C}$ is a non-negative definite self-adjoint operator in $L^{2}(X,m).$ We
address in this paper to the following question:

\emph{How general can be the
spectrum }$\mathsf{Spec}(L_{C})\subseteq\mathbb{R}_{+}?$

When $(X,d)$ is compact, $\mathsf{Spec}(L_{C})$ is an increasing sequence of
eigenvalues of finite multiplicity which contains $0.$ Assuming that $(X,d)$
is not compact we show that under some natural conditions concerning the
structure of the hierarchical lattice ($\equiv$ the tree of $d$-balls) any
given closed subset $S\subseteq\mathbb{R}_{+}$, which contains $0$ as an
accumulation point and is unbounded if $X$ is non-discrete, may appear as $\mathsf{Spec}(L_{C})$ for some appropriately chosen
function $C(B).$ The operator $-L_{C}$ extends to $L^{q}(X,m),$ $1\leq
q<\infty,$ as Markov generator and its spectrum does not depend on $q.$ As an
example, we consider the operator $\mathfrak{D}^{\alpha}$ of fractional
derivative defined on the field $\mathbb{Q}_{p}$ of $p$-adic numbers.


\end{abstract}

\section{Introduction}

\textbf{\ }\setcounter{equation}{0}The concept of hierarchical lattice and
hierarchical distance was proposed by F.J. Dyson in his famous papers on the
phase transition for $\mathbf{1D}$ ferromagnetic model with long range
interaction~\cite{Dyson1,Dyson2}.

The notion of hierarchical Laplacian $L$, which is closely related to the
Dyson's model was studied in several mathematical papers~\cite{Bovier},
\cite{Kritch1,Kritch2,Kritch3} and~\cite{Molchanov}.

These papers contain some basic information about $L$ (the spectrum, Markov
semigroup, resolvent etc) in the case when the hierarchical lattice satisfies
some symmetry conditions (homogeneity, self-similarity etc). Under these
symmetry conditions, $\mathsf{Spec}(L)$ is pure point and all eigenvalues have infinite
multiplicity. The main goal of the papers mentioned above was to introduce a
class of random perturbations of $L$ and then to justify the existence of the
spectral bifurcation from the pure point spectrum to the continuous one.

A systematic study of a class of isotropic Markov semigroups defined on an
ultrametric space $(X,d)$ has been done in~\cite{BGP} (see also the
forthcoming paper~\cite{BGPW}). In particular, given an isotropic Markov
semigroup $(P^{t})$ with Markov generator $-L,$ one can show that the
operator $L$ is a hierarchical Laplacian on $(X,d)$ associated with an
appropriate choice function $C(B)$ and vice versa. Then the general theory
developed in \cite{BGP} and \cite{BGPW} applies: modifying canonically the
underlying ultrametric $d$, we call this new ultrametric $d_{\ast},$ the set
$\mathsf{Spec}(L)$ is pure point and can be described as
\begin{equation}
\mathsf{Spec}(L)=\overline{\left\{  \frac{1}{d_{\ast}(x,y)}:x\neq y\right\}  }%
\cup\{0\}. \label{Spectrum}%
\end{equation}
In our construction the families of $d$-balls and $d_{\ast}$-balls coincide,
whence these two ultrametrics generate the same topology and the same
hierarchical structure, and in particular, the same class of hierarchical Laplacians.

The equation (\ref{Spectrum}) leads us to the following question.

\begin{description}
\item[(A)] \emph{How general can be the set }$\mathsf{Spec}(L)?$
\end{description}

or equivalently,

\begin{description}
\item[(B)] \emph{How general can be the set }$\R(d_{\ast})?$
\end{description}

In the course of study we assume that $(X,d)$ is a locally compact and
separable ultrametric space. Recall that a metric $d$ is called an
\emph{ultrametric} if it satisfies the ultrametric inequality%
\begin{equation}
d(x,y)\leq\max\{d(x,z),d(z,y)\},
\end{equation}
that is obviously stronger than the usual triangle inequality. Usually, we
also assume that the ultrametric $d$ is \emph{proper}, that is, each closed
$d$-ball is a compact set.

\

The paper is organized as follows. In Section 2 we recall some basic
properties of ultrametric spaces. The main original results there can be
summarized in the following statement (see Proposition~\ref{p1},
Theorem~\ref{t1}, Theorem~\ref{t2}).

\begin{theorem}
\label{ultrametric d-d'} Let $(X,d)$ be a locally compact, non-compact,
separable ultrametric space. Let $M\subset\lbrack0,\infty)$ be a countable
unbounded set which contains $0$. Assume that if $X$ contains a non-isolated
point, then $0$ is an accumulation point of $M$. Then the following properties hold:

\begin{enumerate}
\item There exists a proper ultrametric $d^{\prime}$ on $X$ which generates
the same topology as $d$ and such that $\R(d^{\prime})=M.$

\item Assume that $d$ is proper and that there exists a partition $\Pi$ of $X$
made of $d$-balls which contains infinitely many non-singletons. Then the
ultrametric $d^{\prime}$ as above can be chosen such that the collections of
$d$-balls and $d^{\prime}$-balls coincide.
\end{enumerate}
\end{theorem}

Let $\mathcal{B}$ be the set of all non-singleton balls. Let $\mathcal{D}$ be
the set of locally constant functions having compact support. In Section 3,
given a measure $m$ on $X$ which satisfies some natural conditions and a
choice function $C:\mathcal{B}\rightarrow(0,\infty)$, we define (pointwise)
the hierarchical Laplacian $(L_{C},\mathcal{D})$ associated with $m$ and $C,$%
\begin{equation}
L_{C}f(x):=-\sum\limits_{B\in\mathcal{B}:\text{ }x\in B}C(B)\left(
P_{B}f-f(x)\right),\quad f\in\mathcal{D},\label{hlaplacian}%
\end{equation}
where%
\[
P_{B}f:=\frac{1}{m(B)}\int\limits_{B}fdm.
\]
The operator $(L_{C},\mathcal{D})$ acts in $L^{2}=L^{2}(X,m),$ is symmetric
and admits a complete system of eigenfunctions
$\{f_{B,B^{\prime}}\}_{B^{\prime}\in\mathcal{B}},$%
\begin{equation}
f_{B,B^{\prime}}=\frac{1}{m(B)}\mathbf{1}_{B}-\frac{1}{m(B^{\prime}%
)}\mathbf{1}_{B^{\prime}}, \label{eigenfunction}%
\end{equation}
where $B\subset B^{\prime}$ are nearest neighboring
balls; when $m(X)<\infty$, we also set $f_{X,X^{\prime}}=1/m(X)$.  The
eigenvalue $\lambda(B^{\prime})$ corresponding to $f_{B,B^{\prime}}$ is%
\begin{equation}
\lambda(B^{\prime})=\sum\limits_{T\in\mathcal{B}:\text{ }B^{\prime}\subseteq
T}C(T); \label{eigenvalue}%
\end{equation}
when $m(X)<\infty$, we set $\lambda(X^{\prime})=0.$ In particular, we conclude
that $(L_{C},\mathcal{D})$ is an essentially self-adjoint operator in $L^{2}.$
By abuse of notation, we shell write $(L_{C},\mathsf{Dom}_{L_{C}})$ for its
unique self-adjoint extension.

Let $B\subset B^{\prime}$ be two nearest neighboring
balls. Choosing the function
\[
C(B)=\frac{1}{\diam(B)}-\frac{1}{\diam(B^{\prime})}%
\]
we obtain
\[
\lambda(B)=\frac{1}{\diam(B)},
\]
for any $B\in\mathcal{B}.$ Applying Theorem \ref{ultrametric d-d'}, we answer
the question $(A).$

\begin{theorem}\label{generality of spectrum} Let $(X,d)$ be a locally compact, non-compact,
separable ultrametric space. Let $S\subseteq\lbrack0,\infty)$ be a closed set
which contains $0$ as an accumulation point. Assume that if $X$ contains a non-isolated point then $S$
is unbounded. Then the following properties hold:

\begin{enumerate}
\item There exist a proper ultrametric $d^{\prime}$ on $X$ which generates the
same topology as $d$ and a choice function $C(B)$ defined for $(X,d^{\prime})$
such that $\mathsf{Spec}(L_{C})=S.$

\item Assume that $d$ is proper and that there exists a partition $\Pi$ of $X$
made of $d$-balls containing infinitely many non-singletons. Then there exists
a choice function $C(B)$ defined for $(X,d)$ such that $\mathsf{Spec}(L_{C})=S.$
\end{enumerate}
\end{theorem}

Actually, we get Theorem~\ref{generality of spectrum} not only for the particular choice
function mentioned above but for more general types of them (see the proof in Section~3).

A very simple Example~\ref{ex1} shows that the condition \textquotedblleft
there exists a partition $\Pi$ of $X$ made of $d$-balls containing infinitely
many non-singletons" in statement (2) of Theorem~\ref{ultrametric d-d'} and
Theorem~\ref{generality of spectrum} can not be dropped: $X=\mathbb{N}$ and
$d(m,n)=\max(m,n)$ when $m\neq n$ and $0$ otherwise.

In the concluding Section 4 we consider the operator $\mathfrak{D}^{\alpha}$
of the $p$-adic fractional derivative of order $\alpha>0.$ This operator
related to the concept of $p$-adic Quantum Mechanics was introduced by V.S.
Vladimirov, see \cite{Vladimirov}, \cite{VladimirovVolovich} and
\cite{Vladimirov94}. We prove that $\mathfrak{D}^{\alpha}$ is a hierarchical
Laplacian. The main novelty here is that $\mathfrak{D}^{\alpha}$ admits a
closed extension in $L^{q},$ $1\leq q<\infty,$ call it $\mathfrak{D}%
_{q}^{\alpha}.$ The closed operator $-\mathfrak{D}_{q}^{\alpha}$ coincides
with the infinitesimal generator of a translation invariant Markov semigroup
acting in $L^{q}.$ The set $\mathsf{Spec}(\mathfrak{D}_{q}^{\alpha})$ consists
of eigenvalues $p^{k\alpha},k\in\mathbb{Z},$ each of which has infinite
multiplicity and contains $0$ as an accumulation point. In particular,
$\mathsf{Spec}(\mathfrak{D}_{q}^{\alpha})=\mathsf{Spec}(\mathfrak{D}%
_{2}^{\alpha}),$ for all $1\leq q<\infty.$ We study also random perturbations
$\mathfrak{D}^{\alpha}(\omega)$ of the operator $\mathfrak{D}^{\alpha}$ and
provide a limit behaviour of its normalized eigenvalues.

\

\begin{acknowledgement}
We are grateful to A. Grigor'yan, S. Molchanov, Ch. Pittet and W. Woess for
their various comments and encouragement. We thank the participants of
our seminar M. Krupski and W. Cygan for many fruitful discussions.
\end{acknowledgement}

\section{Metric matters}

\setcounter{equation}{0}

Recall that a topological space $X$ is \emph{totally disconnected} if for any
two distinct points $x,y\in X$ there exists a closed and open (=\emph{clopen})
subset $U$ of $X$ such that $x\in U$ and $y\notin U$; if $X$ has a basis
consisting of clopen subsets, then $X$ is called \emph{zero-dimensional}.

Clearly, zero-dimensional spaces are totally disconnected but there are Polish
spaces which are totally disconnected and not zero-dimensional (for example,
the complete Erd\H{o}s space $E=\{(x_{i})\in l_{2}: x_{i}\in\mathbb{R}%
\setminus\mathbb{Q}\}$, see~\cite[Section 1.4]{Eng} ). Nevertheless, for
locally compact Hausdorff spaces these two notions coincide, i.e. totally
disconnected locally compact Hausdorff spaces are zero-dimensional.

At the beginning let us recall two classical topological characterizations
which are crucial in the study of zero-dimensional separable metric spaces
(see, e.g., \cite[p. 35]{Kech}).

\begin{proposition}
\label{p2} A topological space $X$ is homeomorphic to the Cantor set
$\mathcal{C}=\{0,1\}^{\aleph_{0}}$ if and only if $X$ is metrizable, compact,
perfect and totally disconnected.
\end{proposition}

It is well known that each non-empty locally compact, non-compact, metrizable,
separable space $X$ can be compactified by adding an extra point $\omega$
whose neighborhoods are declared to be of the form $\{\omega\}\cup(X\setminus
K)$, where $K$ is a compact subset of $X$. One can easily check that $\omega$
has a countable basis of neighborhoods of this form. It follows that the
compact space $X\cup\{\omega\}$ has a countable base, so it is metrizable. If,
additionally, $X$ has no isolated points (i.e., $X$ is perfect) and is totally
disconnected, then $X\cup\{\omega\}$ is homeomorphic to $\mathcal{C}$ by
Proposition~\ref{p2}. Thus we get the following characterization.

\begin{proposition}
\label{p3} Each metrizable locally compact, non-compact, separable, perfect,
totally disconnected space is homeomorphic to $\mathcal{C}\setminus\{p\}$,
where $p$ is an arbitrary point of $\mathcal{C}$.
\end{proposition}

Let us now list some basic properties of ultrametric spaces $(X,d)$ (see
\cite[p. 227]{B},~\cite{Lem}).

\begin{description}
\item[(a)] \emph{Each ball $B_{r}(a)=\{x\in X: d(x,a)<r\}$ is a closed set,
$\diam B_{r}(a)\le r$ and $d(x,y)=\inf\{\diam B: \text{$B$ is a ball
containing $\{x,y\}$}\}$}.

\item[(b)] \emph{If $x\in B_{r}(a)$, then $B_{r}(a)=B_{r}(x)$.}

\item[(c)] \emph{If two balls intersect, then one of them is contained in the
other.}

\item[(d)] \emph{No infinite ultrametric space is isometric to a subset of the
Euclidean space}~\cite[Proposition 3.1]{Lu}.
\end{description}

If an ultrametric space $(X,d)$ is separable, then the following facts also hold.

\begin{description}
\item[(e)] \emph{The set $\R(d)$ of all values of metric $d$ is at most
countable.}

\item[(f)] \emph{All distinct balls of a given radius $r>0$ form at most countable
partition of $X$.}
\end{description}

If, in addition, $X$ is compact, then

\begin{description}
\item[(g)] \emph{$\R(d)$ has at most one accumulation point which is equal to 0.}
\end{description}

It is easy to see that properties (b) and (c) are in fact characteristic for
an ultrametric:

\begin{itemize}
\item \emph{if $d$ is a metric on $X$ satisfying either one of them then $d$
is an ultrametric.}
\end{itemize}

It follows from the above properties that each ultrametric space has a basis
of clopen sets, i.e., it is zero-dimensional. Conversely,

\begin{itemize}
\item \emph{each zero-dimensional separable metrizable space $X$ is metrizable
by an ultrametric $d$.}
\end{itemize}

It can be defined in the following way. Let $\{B_{1}, B_{2},\dots\}$ be a
basis of clopen subsets of $(X,d)$ and let $f_{n}:X\to\{0,1\}$ be the
characteristic function of $B_{n}$, $n=1,2,\dots$. Then
\[
d(x,y)=\max\{\frac{|f_{n}(x)-f_{n}(y)|}{n}: n\in\mathbb{N}\}
\]
(see \cite{deG}).

\begin{definition}
A metric $d$ on a set $X$ is called proper if every closed $d$-ball
$\overline{B}_{r}(a):=\{x\in X: d(x,a)\leq r\}$ is compact. A metric space
$(X,d)$ is proper if $d$ is proper.
\end{definition}

Notice that any proper metric is complete. It is known that any metrizable,
locally compact, separable space admits a proper metric~\cite{V}.

The proper ultrametric on $\mathbb{N}$ given in the next example is, in a
sense, generic for metrizable, locally compact, non-compact, separable totally
disconnected spaces.

\begin{example}
\label{ex1} Let $\mathbb{N}$ be equipped with the ultrametric
\[
d_{\max}(m,n)=\left\{
\begin{array}
[c]{ll}%
\max\{m,n\}, & \hbox{if $m\neq n$;}\\
0, & \hbox{if $m=n$.}
\end{array}
\right.
\]
Then any $d_{\max}$-ball is either a singleton or is of the form
$\{1,2,\dots,n\}$.
\end{example}

\begin{proposition}
\label{p1} If $X$ is a metrizable, locally compact, separable, totally
disconnected space, then $X$ admits a proper ultrametric that generates the
topology of $X$.
\end{proposition}

\begin{proof}
There is nothing to prove if $X$ is compact, since any ultrametric metrizing
$X$ is automatically proper. Assume that $X$ is not compact. Then there is a
partition $\Pi=\{P_{1},P_{2},\dots\}$ of $X$ made of non-empty, compact-open
subsets of $X$. Let $d$ be an ultrametric on $X$ that generates the topology
of $X$. We get our proper ultrametric by the following formula:%

\[
d_{\Pi}(x,y) = \left\{
\begin{array}
[c]{ll}%
d(x,y), & \hbox{if $x,y\in P_k$ for some  $k$;}\\
\max(m,n), & \hbox{if $x\in P_m, y\in P_n$ and $m\neq n$.}
\end{array}
\right.
\]
Indeed, observe that $d_{\Pi}(P_{1},P_{n}):=\inf\{d(x,y): x\in P_{1}, y\in
P_{n}\}\rightarrow\infty$ if $n\rightarrow\infty$. Hence, each closed $d_{\Pi%
}$-bounded subset of $X$ is compact.
\end{proof}

\begin{example}
\label{ex p-adic} One of the most known example of a proper ultrametric space
is the field $\mathbb{Q}_{p}$ of $p$-adic numbers endowed with the $p$-adic
norm $\left\Vert x\right\Vert _{p}$ and the $p$-adic ultrametric $d(x,y)=$
$\left\Vert x-y\right\Vert _{p}$.

This ultrametric space is locally compact, non-compact, separable, perfect and
totally disconnected, so it is homeomorphic to the Cantor set minus a point
(see Proposition~\ref{p3}).
\end{example}

\begin{example}
\label{ex cyclic} The other example which we have in our mind is a discrete
Abelian group
\[
G=\mathbb{Z}(p_{1})\oplus\mathbb{Z}(p_{2})\oplus...\mathbb{\ },
\]
the infinite sum of cyclic groups $\mathbb{Z}(p_{k})=\mathbb{Z}/p_{k}%
\mathbb{Z}$, where $(p_{k})$ is a given sequence of integers. The ultrametric
$d$ in $G$ is defined as follows: $d\left(  x,y\right)  $ is the minimal value
of $n$ such that $x$ and $y$ belong to the same coset of the subgroup
$G_{n}=\mathbb{Z}(p_{1})\oplus\mathbb{Z}(p_{2})\oplus...\oplus\mathbb{\ Z}%
(p_{n}).$
\end{example}

\begin{lemma}
\label{partition} Let $X$ be a metrizable, locally compact, non-compact,
separable, totally disconnected space. There exists an infinite countable
partition $\Pi$ of $X$ such that each element of $\Pi$ is a nondegenerate
(i.e., non-singleton) compact-open subset of $X$ which contains an
accumulation point or, else, it is a two-point set.
\end{lemma}

\begin{proof}
Let us take any infinite, countable partition $\mathcal{R}$ of $X$ consisting
of compact-open subsets of $X$ and let
\[
\mathcal{R}^{\prime}=\{R\in\mathcal{R}: \text{$R$ contains no accumulation
points}\}.
\]
Clearly, every $R\in\mathcal{R}^{\prime}$ is a finite set. If $\mathcal{R}%
^{\prime}$ is finite, then choose $R_{0}\in\mathcal{R}\setminus\mathcal{R}%
^{\prime}$ and define
\[
\Pi=\{R_{0}\cup\bigcup\mathcal{R}^{\prime}\}\cup(\mathcal{R}\setminus
(\mathcal{R}^{\prime}\cup\{R_{0}\})).
\]

If $\mathcal{R}^{\prime}$ is infinite, say $\mathcal{R}^{\prime}=\{R_{1}%
,R_{2},\dots\}$, then there are mutually disjoint two-point sets $P_{1}%
,P_{2},\dots$ such that $\bigcup_{n=1}^{\infty}R_{n}=\bigcup_{n=1}^{\infty
}P_{n}$ and we can put
\[
\Pi=(\mathcal{R}\setminus\mathcal{R}^{\prime})\cup\{P_{1},P_{2},\dots\}.
\]
The proof is finished.
\end{proof}

\

Notice that if $(X,d)$ is a proper ultrametric space then, for any increasing
function $\phi:\R(d)\rightarrow\lbrack0,\infty)$, metric $d^{\prime}=\phi\circ
d$ is again a proper ultrametric having the same collection of balls as $d$.
Hence, by Proposition~\ref{p1}, each nondegenerate, metrizable, locally
compact, separable, totally disconnected space admits infinitely many proper
equivalent ultrametrics.

\

Let $(X,d)$ be a compact separable ultrametric space. It follows from property
(g) that if $X$ is finite, then $\R(d)$ is finite and if $X$ is infinite, we
can arrange the values of $d$ in a sequence decreasing to 0, i.e.,
$\R(d)=\{c_{1},c_{2},\dots\}$, where $c_{n}\searrow0$. Since, for any other
sequence $c_{n}^{\prime}\searrow0$, there is an increasing surjection
$\phi:\{c_{1},c_{2},\dots\}\rightarrow\{c_{1}^{\prime},c_{2}^{\prime},\dots
\}$, we get equivalent proper ultrametric $d^{\prime}=\phi\circ d$ on $X$ with
the same family of balls as $d$ and with $\R(d^{\prime})=\{c_{1}^{\prime
},c_{c}^{\prime},\dots\}$.

An analogous statement for locally compact, non-compact $(X,d)$ has a more
complicated nature.

\begin{theorem}
\label{t1} Let $X$ be a metrizable, locally compact, non-compact, separable,
totally disconnected space and $\Pi$ be an infinite countable partition of $X$
as in Lemma~\ref{partition}. Let $M\subset[0,\infty)$ be a countable
unbounded set containing $0$. Assume that if $X$ contains an accumulation
point, then $0$ is an accumulation point of $M$. Then, for each proper
ultrametric $d_{p}$ which generates the topology of $X$, there exists an
equivalent proper ultrametric $d$ on $X$ such that $\R(d)=M$ and the
collections of $d_{p}$-balls and $d$-balls contained in any $P\in\Pi$
coincide. Moreover, $d(x,y)\le d_{p}(x,y)$ for $x,y\in B$ and each $d_{p}%
$-ball $B\varsubsetneq P\in\Pi$.
\end{theorem}

The idea of the proof of Theorem~\ref{t1} is based on a specific tree-structure
of the family of balls in an ultrametric space. So, let us first introduce
necessary notions.

Let $\mathcal{T}(X,d)$ be the collection of all balls in a proper ultra-metric
space $(X,d)$. Consider $\mathcal{T}(X,d)$ with a partial order
\[
B^{\prime}\preccurlyeq B\Leftrightarrow B\subseteqq B^{\prime}.
\]
Observe that, by properties (c) and (g), each ball $B\in\mathcal{T}(X,d)$ has
a unique immediate predecessor with respect to $\preccurlyeq$ and if $B$ is
not a singleton, then it has at most finitely many immediate successors. If
$A\wedge B=\inf\{A,B\}$, the infimum taken with respect to $\preccurlyeq$
(which is the smallest, with respect to the inclusion, ball containing balls
$A$ and $B$), then $(\mathcal{T}(X,d),\wedge)$ is a semi-lattice. We prefer to
view $\mathcal{T}(X,d)$ geometrically as a graph with vertices being elements
of $\mathcal{T}(X,d)$ and edges being pairs of $d$-balls $(B,B^{\prime})$ such
that $B$ is an immediate successor or predecessor of $B^{\prime}$.

A path in $\mathcal{T}(X,d)$ from $B_{1}$ to $B_{n}$ is a finite sequence
$B_{1}, B_{2},\dots, B_{n}$ of mutually distinct vertices such that, for each
$i=1,2,\dots,n-1$, either $(B_{i},B_{i+1})$ or $(B_{i+1},B_{i})$ is an edge.
Given two vertices $A,B$, there are unique paths from $A$ to $A\wedge B$ and
from $B$ to $A\wedge B$ and the concatenation of these two paths gives the
unique path from $A$ to $B$. Thus, $\mathcal{T}(X,d)$ is a countable, locally
finite, path-connected tree. Vertices with no successor are called end-points
of the tree; they represent singleton balls.

Let $2^{Y}$ be the family of compact nonempty subsets of a Hausdorff
topological space $Y$. We consider $2^{Y}$ with the Vietoris topology (which
is generated by the subbase of sets of the form $\{A\in2^{Y}: A\subset U\}$
and $\{A\in2^{Y}: A\cap U\neq\emptyset\}$ whenever $U$ is open in $Y$). Recall
that if $(Y,d)$ is a metric space, then $2^{Y}$ is metrizable by the Hausdorff
metric
\[
H_{d}(A,B):= \inf\{r>0: \text{$A\subset N_{d}(r,B)$ and $B\subset N_{d}(r,A)$%
}\},
\]
where $N_{d}(r,A)= \{x\in Y: d(x,A)<r\}$, $d(x,A)=\inf\{d(x,a): a\in A\}$
(see, e.g.,~\cite{IN}).

\begin{definition}\label{def Whitney}
Given $\mathcal{H}\subset2^{Y}$, a continuous map
$w:\mathcal{H}\to[0,\infty)$ is said to be a Whitney map for $\mathcal{H}$ if
\begin{enumerate}
\item $w(A)<w(B)$ whenever $A\varsubsetneq B$,
\item $w(A)=0$ if and only if $A$ is a singleton.
\end{enumerate}
\end{definition}

Whitney maps for $2^{Y}$ exist for metric separable spaces $Y$~\cite[p.
205]{IN}. It is easy to see that the diameter function $\diam$ is never a
Whitney map for $\mathcal{H}=\{F\subset Y: |F|\le3\}$ (thus for $2^{Y}$ as
well) if $Y$ contains at least three points. Nevertheless, if $(Y,d)$ is a
proper ultrametric space, then property (b) implies that $\diam:\mathcal{T}%
(Y,d)\to[0,\infty)$ is a Whitney map.

Denote $F_1=\{\{x\}: x\in X\}$. The following proposition will be used in Section~\ref{HierLapl}.

\begin{proposition}\label{WhitneyLemma}
Let $(X,d)$ be a proper ultrametric space with a Whitney map $w: \mathcal{T}(X,d)\cup F_1\to [0,\infty)$ satisfying the following condition:
\begin{equation}\label{eq_proper}
\lim_{n\to\infty} w(B_n) =\infty \quad\text{for each infinite sequence}\quad B_1\varsubsetneq B_2\varsubsetneq\dots.
\end{equation}
Then the formula
$ d_{\ast}(x,y)= w(\{x\}\wedge \{y\})$ for $x\neq y$ (where $\{x\}\wedge \{y\}$ denotes the smallest ball containing $x$ and $y$) and $d_{\ast}(x,y)= 0$ for $x=y$
  defines a proper ultrametric in $X$ which induces the same topology and the same collection of balls as $d$.
\end{proposition}

\begin{proof}
It is easy to verify that $d_{\ast}$ is an ultrametric. It is equivalent to $d$
by the continuity of $w$ at each singleton. In order to show that $d_{\ast}$ is
proper, let $\overline{B^*_r}(x)$ be a closed $d_{\ast}$-ball of radius $r$ centered
at $x$.  Notice that~\eqref{eq_proper} implies that there is a $d$-ball $B(x)$,
centered at $x$, containing
 $\overline{B^*_r}(x)$. Since $B(x)$ is compact and $B^*_r(x)$ is closed, we conclude
 that the ball $\overline{B^*_r}(x)$ is compact.

\begin{claim}\label{claim_balls}
For any $d$-ball $B(x)$ centered at $x$ there exists $\epsilon>0$ such that
$d_{\ast}$-ball $B^*_{r(\epsilon)}(x)$ of radius $r(\epsilon)=w(B(x))+\epsilon$,
centered at $x$, is equal to  $B(x)$.
\end{claim}

\begin{proof}[Proof of Claim~\ref{claim_balls}]
Clearly, we can assume that $B(x)$ is nondegenerate. The inclusion
$B(x)\subset B^*_{r(\epsilon)}(x)$ holds for each $\epsilon>0$, since if
$y\in B(x)\setminus \{x\}$, then $\{x\}\wedge \{y\}\subset B(x)$, hence
\[
d_{\ast}(x,y)=w(\{x\}\wedge \{y\})\le w(B(x))< w(B(x))+\epsilon.
\]
Suppose that, for each $\epsilon>0$, $B(x)$ is a proper subset of
$B^*_{r(\epsilon)}(x)$. For each $n\in \mathbb N$, choose a point
$y_n\in B^*_{r(\frac1n)}(x)\setminus B(x)$. Then
$B(x)\varsubsetneq \{x\}\wedge \{y_n\}$ and we get

\begin{equation}\label{eq_claim}
0<w(B(x))< w(\{x\}\wedge \{y_n\})=d_{\ast}(x,y_n)< w(B(x))+\frac1n.
\end{equation}
The balls $\{x\}\wedge \{y_n\}$ are contained in a branch of $\mathcal{T}(X,d)$, so
we can choose a subsequence $n_k$ such that  balls
 $\{x\}\wedge \{y_{n_k}\}$ form a decreasing family of sets, in view
of~\eqref{eq_claim}. So $d_{\ast}(x,y_{n_k}) \to w(B(x))$ if $k\to\infty$ and points
 $x, y_{n_k}$, $k\in\mathbb N$,  belong to a compact set $\{x\}\wedge \{y_{n_1}\}$.
It means that $\R(d_{\ast})$ has an accumulation point different from 0 on a compact
set, contrary to property (g).
\end{proof}

It remains to show that each nondegenerate $d_{\ast}$-ball $B^*(x)$ centered at $x$
is a $d$-ball. Let $Z:= \{r>0: B^*(x)= B^*_r(x)\}$. Obviously $Z\neq\emptyset$.
Denote $r'=\inf Z$. Notice that $r'>0$. We have
\begin{equation}\label{eq_balls}
B^*(x)=\{y: d_{\ast}(x,y)\le r'\}.
\end{equation}
Indeed, the inclusion $\subset$ in \eqref{eq_balls} is obvious for any metric.  The inclusion $\supset$  is also trivial in the case $r'\notin Z$.
But the case $r'\in Z$  cannot occur because $B^*(x)=B^*_{r'}(x)$ would mean that $r'$ is an accumulation point of distances $d_{\ast}(x,y)$ for points $y\in B^*(x)$ in the compact set $\overline{B^*_{r'}}(x) = \{y: d_{\ast}(x,y)\le r'\}$, contrary to property (g). This also shows that
\begin{equation}\label{eq_distance}
\text{there is $a\in B^*(x)$ satisfying $ d_{\ast}(x,a)=r'$.}
\end{equation}
Finally, we can see that
\begin{equation}\label{eq_wedge}
 B^*(x)=\{x\}\wedge \{a\}\in \mathcal{T}(X,d).
\end{equation}
In fact, we have already shown that the $d$-ball $\{x\}\wedge \{a\}$ is also a
$d_{\ast}$-ball, so either  $\{x\}\wedge \{a\}\subset B^*(x)$ or
$B^*(x) \varsubsetneq \{x\}\wedge \{a\}$.

Suppose the latter inclusion holds. Then, by~\eqref{eq_balls}
and ~\eqref{eq_distance},  there is a point $y\in \{x\}\wedge \{a\}$ with
$d_{\ast}(x,y)> r'=d_{\ast}(x,a)$. Hence,
$w(\{x\}\wedge \{y\})> w(\{x\}\wedge \{a\})$, so
$\{x\}\wedge \{y\}\varsupsetneq\{x\}\wedge \{a\}$. On the other hand,
$y\in \{x\}\wedge \{a\}$ implies that $\{x\}\wedge \{y\}\subset \{x\}\wedge \{a\}$,
 a contradiction.

Therefore  $\{x\}\wedge \{a\}\subset B^*(x)$. But the inclusion cannot be strict,
since otherwise there is  a point $y\in B^*(x)\setminus \{x\}\wedge \{a\}$ and we
have
\[
w(\{x\}\wedge \{y\})=d_{\ast}(x,y)\le r'=d_{\ast}(x,a)=w(\{x\}\wedge \{a\}),
\]
which implies $\{x\}\wedge \{y\}\subset \{x\}\wedge \{a\}$, a contradiction. The proof of Proposition~\ref{WhitneyLemma} is finished.

\end{proof}

\

\begin{proof}[Proof of Theorem~\ref{t1}]
We are going to construct a countable, locally
finite, path-connected tree $\mathcal{T}\subset2^{X}$ without the least
element and a Whitney map $w$ for $\mathcal{T}\cup F_1$ such that $\R(w)=M$ (recall that $F_1$ is the set of all singletons in $X$).

Let $\Pi=\{P_{1},P_{2},\dots\}$. Embed the tree $\mathcal{T}(\mathbb{N}%
,d_{\max})$ from Example~\ref{ex1} in $2^{X}$ by a one-to-one function $\phi$
such that $\phi(\{n\})=P_{n}$ and $\phi(\{1,2,\dots,n\})=\bigcup_{i=1}^{n}
P_{i}$.

For each $P\in\Pi$, let $\mathcal{T}(P,d_{p})$ be the rooted (at $P$) tree of
closed $d_{p}$-balls of $X$ which are contained in $P$. Observe that if $P$ is
a doubleton, then $\mathcal{T}(P,d_{p})$ splits into two singletons, otherwise
$\mathcal{T}(P,d_{p})$ has an infinite branch contained in the set
$\widetilde{\mathcal{T}(P,d_{p})}$ of non-singleton vertices. Trees
$\mathcal{T}(P,d_{p})$ extend the tree $\phi(\mathcal{T}(\mathbb{N},d_{\max
}))$ and as a result we obtain a tree $\mathcal{T}\subset2^{X}$.

To each branch $\mathcal{L}$ of $\mathcal{T}$ there corresponds a point
$x_{\mathcal{L}}\in X$ such that $\{x_{\mathcal{L}}\}= \bigcap\mathcal{L}$
(since $(X,d_{p})$ is a complete space) and this correspondence is a bijection
between the set of all branches and $X$. Observe that the bijection locally
(inside of each $P\in\Pi$) preserves $d_{p}$-balls. The semi-lattice operation
$\wedge$ on $\mathcal{T}$ can now be extended over the set of all singletons
of $X$ by
\[
\{x_{\mathcal{L}}\}\wedge\{x_{\mathcal{L}^{\prime}}\}=\inf({\mathcal{L}}%
\cap{\mathcal{L}^{\prime}}).
\]

We will now construct a Whitney map $w$ for $\mathcal{T}\cup F_1$ such that $\R(w)=M$.
Enumerate positive numbers in $M$ as $m_{1},m_{2},\dots$ and choose a sequence
$(k_{n})\in M$ such that%

\[
k_{1}>\max\{m_{1},m_{2}\},\quad k_{n+1}>\max\{k_{n},m_{n+1}\}
\]

(this can be done since $M$ is unbounded).

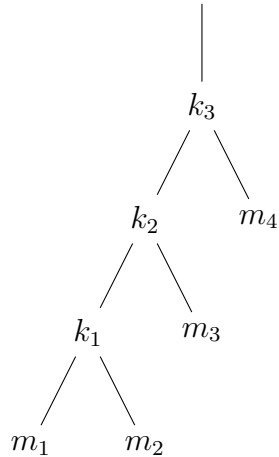
\begin{center}

%\pstree[radius=10pt,levelsep=1cm]{\Tp}{\pstree{\TR{$k_3$}}{\pstree{\TR{$k_2$}}{\pstree{\TR{$k_1$}}{\TR{$m_1$}\TR{$m_2$}}\TR{$m_3$}}\TR{$m_4$}}}

\begin{tikzpicture}[level/.style={level distance = 1.5cm}]

\node  {}
    child{ node  {$k_3$}
            child{ node  {$k_2$}
            	child{ node  {$k_1$}
			child{ node  {$m_1$}}
			child{ node  {$m_2$}}
		}
		child{ node  {$m_3$}}
	}
	child{ node  {$m_4$}}
	};
\end{tikzpicture}

\end{center}

\

Put $w(B)=0$ if $B$ is a singleton, $w(B)=m_{n}$ if $B=P_{n}$ and $w(B)=k_{n}$
if $B=\bigcup_{i=1}^{n} P_{i}$. It remains to define $w$ on the set
$\widetilde{\mathcal{T}(P,d_{p})}$ for each $P\in\Pi$ that contains an
accumulation point. Then, since $0$ must be an accumulation point of $M$ by
the hypothesis, there is a strictly decreasing sequence $w(P)>w_{n}%
\rightarrow0$ in $M$. Now, take a function
\[
s_{P}: \widetilde{\mathcal{T}(P,d_{p})}\setminus\{P\}\to\{w_{1},w_{2}%
,\dots\}
\]
which is order preserving on each branch of $\widetilde{\mathcal{T}(P,d_{p})}$, i.e., $s_{P}(B)< s_{P}(B^{\prime})$ if $B\varsubsetneq B^{\prime}$, such that $s_{P}(B)\leq \diam_{d_{p}}(B)$ and put
\[
w(B)=s_{P}(B)\quad\text{ for}\quad B\in\widetilde
{\mathcal{T}(P,d_{p})}\setminus\{P\}.
\]
The construction of the Whitney map $w$ is thus complete.

The desired ultrametric $d$ is given by the formula
\[
d(x,y) = w(\{x\}\wedge\{y\}) \quad\text{for}\quad x\neq y \quad\text{and}\quad d(x,y)= 0 \quad\text{for}\quad x=y.
\]

Notice that $w(B) = \diam B$ in the metric $d$ and $\mathcal{T}=\mathcal{T}%
(X,d)$.
\end{proof}

\begin{remark}
\emph{As we have already remarked in the proof of Theorem~\ref{t1}, the trees
$\mathcal{T}(X,d_{p})$ and $\mathcal{T}(X,d)$ locally coincide, i.e., the
collections of $d_{p}$-balls and $d$-balls are the same within each $P\in\Pi$.
Whether one can build an ultrametric $d$ on $X$ which satisfies conditions of
the theorem and such that collections of all $d_{p}$-balls and $d$-balls
coincide is an interesting on its own and useful in applications question, see
Section "Hierarchical Laplacian". Example~\ref{ex1} shows that, in general,
the answer is negative. On the other hand, the answer is positive under the
following extra condition:}

There is a partition $\Pi$ of $X$ consisting of $d_{p}$-balls and infinitely
many of the balls are non-singletons. \emph{In terms of the order
$\preccurlyeq$:} there is an infinite antichain in $\mathcal{T}(X,d_{p})$
(i.e., a subset of $\mathcal{T}(X,d_{p})$ whose elements are pairwise
incomparable by $\preccurlyeq$) which contains at most finitely many end-points.

\emph{Notice that a maximal antichain in $\mathcal{T}(X,d_{p})$ is a partition
of $X$.}
\end{remark}

The above condition evidently holds if the ultrametric space $X$ is perfect
(or contains at most finitely many isolated points).

The following example (a particular case of Example~\ref{ex cyclic}) is a good
illustration of the condition in case of discrete $X$.

\begin{example}
\label{ex2} Consider the infinite countable discrete group
\[
X=\bigoplus_{k=1}^{\infty}\mathbb{G}_{k}, \quad\mathbb{G}_{k}=\mathbb{Z}(2)
\]
with the standard ultrametric
\[
d_{p}(x,y)=\min\{k: x-y\in G_{k}\},
\]
where
\[
G_{0}=\{0\},\quad G_{k}=\bigoplus_{1\le i\le k} \mathbb{G}_{i}.
\]
All $d_{p}$-balls are either finite subgroups $G_{k}$ or their cosets
$G_{k}+g$. The balls form a binary tree $\mathcal{T}(X)$ without the least
element and with singletons as its end-points.
\end{example}

\begin{lemma}
\label{partition1} Let $(X,d_{p})$ be a separable proper ultrametric space.
Suppose there is a partition $\mathcal{S}$ of $X$ consisting of $d_{p}$-balls
and infinitely many of the balls are nondegenerate. Then there is a partition
$\Pi$ consisting of $d_{p}$-balls with infinitely many nondegenerate members
$P$ such that $P$ either contains an accumulation point or all immediate
$\preccurlyeq$-successors of $P$ are singletons.
\end{lemma}

\begin{proof}
Let $\mathcal{B}=\{B_{1},B_{2},\dots\}\subset\mathcal{S}$ be the family of all
nondegenerate elements of $\mathcal{S}$.

We modify the partition $\mathcal{S}$ as follows. For each $B\in\mathcal{B}$
which contains no accumulation point (i.e., $B$ is finite), choose a point
$b\in B$ and let $B(b)\subset B$ be a ball which is an immediate
$\preccurlyeq$-predecessor of $\{b\}$. The modified partition $\Pi$ consists
of all elements of $\mathcal{S}$ which contain an accumulation point, all
balls of the form $B(b)$ and all remaining singletons.
\end{proof}

\begin{theorem}
\label{t2} Let $(X,d_{p})$ be a separable proper ultrametric space. Suppose
there is a partition of $X$ consisting of $d_{p}$-balls and infinitely many of
the balls are non-singletons. Then, for every set $M\subset[0,\infty)$
satisfying the hypotheses of Theorem~\ref{t1}, there is an equivalent proper
ultrametric $d$ on $X$ such that $\R(d)=M$, the collections of $d_{p}%
$-balls and $d$-balls coincide and $d\le d_{p}$ on
balls which are  proper subsets of those elements of the partition that contain  accumulation points.
\end{theorem}

\begin{proof}
By Lemma~\ref{partition1} there is a partition $\Pi$ such that each
nondegenerate element $P\in\Pi$ either contains an accumulation point or all
immediate $\preccurlyeq$-successors of $P$ are singletons. Let $\{B_{1}%
,B_{2},\dots\}\subset\Pi$ be the family of all nondegenerate elements of $\Pi$.

We slightly modify the proof of Theorem~\ref{t1} by considering the original
tree of $d_{p}$-balls over partition $\Pi$ instead of tree $\phi
(\mathcal{T}(\mathbb{N},d_{\max}))$.

Let $0=l_{0}< l_{1}< l_{2}\dots\to\infty$ be a sequence such that
\[
M_{k}:= M\cap(l_{k-1},l_{k}]\neq\emptyset\quad\text{for each $k>0$}.
\]
Consider a function $\kappa:M\to\mathbb{N}$ such that $\kappa(m)$ is the
(unique) index satisfying $m\in M_{\kappa(m)}$.

Let $M\setminus\{0\}=\{m_{1},m_{2},\dots\}$.

Let us define a Whitney map $w$ for $\mathcal{T}(X,d_{p})\cup F_1$. Put $w=0$ for
all singletons  and let $w(B_{i})= m_{i}$. Each $d_{p}$-ball
$B$ preceding some $P\in\Pi$ uniquely decomposes into the union of distinct
elements of $\Pi$ (one of them is $P$ itself):
\[
B=P_{i_{1}}\cup\dots\cup P_{i_{n}} \quad\text{for some $i_{1}<\dots< i_{n}$}.
\]
For such ball $B$, choose $w(B)$ as a number in $M_{s(B)}$, where
\[
s(B)=\sum_{t=1}^{n} \kappa(w(P_{i_{t}})).
\]

If a non-singleton ball $B$ succeeds a $P\in\Pi$ in $\mathcal{T}(X,d_{p})$,
then $P\in\mathcal{B}$ and $B\in\widetilde{\mathcal{T}(P,d_{p})}$, so we can
define $w(B)$ as in the the proof of Theorem~\ref{t1}. The metric
\[
d(x,y) = w(\{x\}\wedge\{y\})
\quad\text{if}\quad x\neq y \quad\text{and}\quad d(x,y)=0 \quad\text{if}\quad x=y
\]
is the required one.

We can also observe that $d(x,y)\le d_{p}(x,y)$ for $x,y\in B$ and each
$d_{p}$-ball $B$ properly contained in $P\in\Pi$.
\end{proof}

\begin{remark}
\emph{We can compare Theorem~\ref{t1} with a result in~\cite[Theorem 2]{Lem}
which says that by a slight change of an ultrametric $d$ in an arbitrary
separable ultrametric space $(X$,d) one can get an equivalent ultrametric
$r\leq d$ that assumes only dyadic rational values. No preservation of balls
is discussed in~\cite{Lem}.}
\end{remark}

\section{Hierarchical Laplacian}\label{HierLapl}

\setcounter{equation}{0}The aim of this section is to justify the properties
of the hierarchical Laplacian listed in the Introduction. Let $(X,d)$ be a
locally compact, separable, proper ultrametric space. Let $m$ be a Radon
measure on $X$ such that $m(B)>0$ for each ball $B$ of positive diameter and
such that $m(\{x\})=0$ if and only if $x$ is a non-isolated point. Let
$\mathcal{D}$ be the set of locally constant functions having compact support.

\begin{definition}\label{choice function}
A choice function $C(B)$ is a function defined on the set $\mathcal{B}$ of all
non-singleton balls $B$, taking values in $(0,\infty)$ and such that%
\begin{enumerate}
\item $\lambda(B):=\sum\limits_{T\in\mathcal{B}:\text{ }B\subseteq
T}C(T)<\infty,$

\item $\lim_{B\downarrow\{x\}}$ $\lambda(B)$ $=\infty$ if $x$ is not an
isolated point.

\end{enumerate}
\end{definition}

Given a choice function $C(B)$ and a measure $m$ as above we consider the
hierarchical Laplacian $(L_{C},\mathcal{D)}$ defined pointwise by the equation
(\ref{hlaplacian}), that is,%
\[
L_{C}f(x):=-\sum\limits_{B\in\mathcal{B}:\text{ }x\in B}C(B)\left(
P_{B}f-f(x)\right)  .
\]

\begin{lemma}
\label{L-C-domain}$(L_{C},\mathcal{D)}$ acts in all spaces $L^{p},$ $1\leq
p\leq\infty.$
\end{lemma}

\begin{proof}
Since the intersection $L^{1}\cap L^{\infty}$ is a subset of each $L^{p}$, $p>1$,
  it is enough to prove the claim if $p$ equals $1$ and $\infty.$
For any ball $T$ of positive measure we set $f_{T}=\mathbf{1}_{T}/m(T)$ and
compute $L_{C}(f_{T})(x),$%

\begin{align*}
L_{C}(f_{T})(x)  &  =-\sum\limits_{x\in B}C(B)\left(  P_{B}(f_{T}%
)-f_{T}(x)\right) \\
&  =-\left(  \sum\limits_{x\in B,T\subseteq B}+\sum\limits_{x\in B,B\cap
T=\varnothing}+\sum\limits_{x\in B,B\subset T}\right)  C(B)\left(  P_{B}%
(f_{T})-f_{T}(x)\right)  .
\end{align*}
Next observe that for any ball $B$ centered at $x,$%
\[
P_{B}\left(  f_{T}\right)  =\left\{
\begin{array}
[c]{ccc}%
f_{T}(x) & \text{if} & B\subseteq T\\
f_{B}(x) & \text{if} & T\subseteq B\\
0 & \text{if} & B\cap T=\varnothing
\end{array}
\right.  .
\]
It follows that%
\begin{align*}
L_{C}(f_{T})(x)  &  =-\sum\limits_{x\in B,T\subseteq B}C(B)\left(
f_{B}(x)-f_{T}(x)\right) \\
&  =\left(  \sum\limits_{x\in B,T\subseteq B}C(B)\right)  f_{T}(x)-\sum
\limits_{x\in B,T\subseteq B}C(B)f_{B}(x).
\end{align*}
Clearly we have%
\[
\left(  \sum\limits_{x\in B,T\subseteq B}C(B)\right)  f_{T}(x)=\lambda
(T)f_{T}(x),
\]
whence%
\begin{equation}
L_{C}(f_{T})(x)=\lambda(T)f_{T}(x)-\sum\limits_{x\in B,T\subseteq B}%
C(B)f_{B}(x). \label{Eq.L-C-f-T}%
\end{equation}
Evidently, $u_{1}=\lambda(T)f_{T}$ is in $L^{1}\cap L^{\infty}.$ For the
second term in (\ref{Eq.L-C-f-T}), call it $u_{2},$ we have%
\begin{align*}
u_{2}(x)  &  =\sum\limits_{B:\text{ }\{x\}\wedge T\subseteq B}C(B)f_{B}%
(x)=\sum\limits_{B:\text{ }\{x\}\wedge T\subseteq B}\frac{C(B)}{m(B)}\\
&  \leq\frac{1}{m(T)}\sum\limits_{B:\text{ }\{x\}\wedge T\subseteq B}%
C(B)\leq\frac{\lambda(T)}{m(T)}.
\end{align*}
Let $T:=T_{0}\subset T_{1}\subset T_{2}\subset...\subset X$ be an increasing
sequence of balls such that each $T_{l+1}$ is the immediate predecessor of
$T_{l}.$ We set $T_{-1}=\varnothing$ and write%
\begin{align*}
\int\limits_{X}u_{2}dm  &  =\sum\limits_{i=0}^{\infty}\int\limits_{T_{i}%
\backslash T_{i-1}}u_{2}dm\\
&  =m(T)\sum\limits_{B:\text{ }T\subseteq B}\frac{C(B)}{m(B)}+\left(
m(T_{1})-m(T)\right)  \sum\limits_{B:\text{ }T_{1}\subseteq B}\frac
{C(B)}{m(B)}+...\text{ .}%
\end{align*}
Applying the Abel transformation we obtain%
\begin{align*}
\int u_{2}dm  &  =\lim_{l\rightarrow\infty}\left[  \sum\limits_{i=0}%
^{l-1}m(T_{i})\left(  \sum\limits_{B:\text{ }T_{i}\subseteq B}\frac
{C(B)}{m(B)}-\sum\limits_{B:\text{ }T_{i+1}\subseteq B}\frac{C(B)}%
{m(B)}\right)  \right. \\
&  \left.  +\text{ }m(T_{l})\sum\limits_{B:\text{ }T_{l}\subseteq B}%
\frac{C(B)}{m(B)}\right] \\
&  =\lim_{l\rightarrow\infty}\left(  \sum\limits_{i=0}^{l-1}C(T_{i}%
)+m(T_{l})\sum\limits_{B:\text{ }T_{l}\subseteq B}\frac{C(B)}{m(B)}\right)
=\lambda(T),
\end{align*}
whence, in particular, $u_{2}$ is in $L^{1}.$ All the above shows that
$L_{C}(f_{T})$ belongs to both $L^{1}$ and $L^{\infty}.$ This finishes the
proof since any locally constant function with compact support is a finite
linear combination of the functions $f_{T}$.
\end{proof}

Let $\{f_{B,B^{\prime}}\}_{B^{\prime}\in\mathcal{B}}$ be the family of
functions defined by the equation (\ref{eigenfunction}), i.e.,%
\[
f_{B,B^{\prime}}=\frac{1}{m(B)}\mathbf{1}_{B}-\frac{1}{m(B^{\prime}%
)}\mathbf{1}_{B^{\prime}}.
\]
It is easy to see that all functions $f_{B,B^{\prime}}\in\mathcal{D}$ and that
for any two distinct balls $B^{\prime}$ and $C^{\prime}$ the functions
$f_{B,B^{\prime}}$ and $f_{C,C^{\prime}}$ are orthogonal in $L^{2}%
=L^{2}(X,m).$

\begin{proposition}
\label{Spectral properties}In the above notation the following properties hold.
\end{proposition}

\begin{enumerate}
\item $\{f_{B,B^{\prime}}\}_{B^{\prime}\in\mathcal{B}}$ is a complete system
in $L^{2}.$

\item $L_{C}(f_{B,B^{\prime}})(x)=\lambda(B^{\prime})f_{B,B^{\prime}}(x),$ for
any $x\in X$ and $B^{\prime}\in\mathcal{B}.$

In particular, $(L_{C},\mathcal{D})$ is a non-negative definite essentially
self-adjoint operator in $L^{2}.$ By abuse of notation, we shell write
$(L_{C},\mathsf{Dom}_{L_{C}})$ for its unique self-adjoint extension.
\end{enumerate}

\begin{proof}
For the first claim, consider $f_{B}=$ $\mathbf{1}_{B}/m(B)$ for any ball $B$
of positive measure and observe that the series%
\begin{equation}
f_{B}=\sum\limits_{T:\text{ }B\subseteq T}f_{T,T^{\prime}} \label{f_series}%
\end{equation}
converges pointwise, and since%
\[
\sum\limits_{T:\text{ }B\subseteq T}\left\Vert f_{T,T^{\prime}}\right\Vert
^{2}=\sum\limits_{T:\text{ }B\subseteq T}\left(  \frac{1}{m(T)}-\frac
{1}{m(T^{\prime})}\right)  =\left\Vert f_{B}\right\Vert ^{2},
\]
the series (\ref{f_series}) converges in $L^{2}$ as well. This evidently
proves the claim.

For the second claim, fix a couple of closest neighbors $T$ $\subset
T^{\prime}$ and write the equation (\ref{Eq.L-C-f-T}) for both $T$ and
$T^{\prime}.$ Subtracting the $T^{\prime}$-equation from the $T$-equation we
obtain%
\begin{align*}
L_{C}(f_{T,T^{\prime}})(x)  &  =L_{C}(f_{T})(x)-L_{C}(f_{T^{\prime}})(x)\\
&  =\lambda(T^{\prime})f_{T}(x)-\lambda(T^{\prime})f_{T^{\prime}}%
(x)=\lambda(T^{\prime})f_{T,T^{\prime}}(x).
\end{align*}
as desired.

The operator $(L_{C},\mathcal{D})$ acts in $L^{2}$ by Lemma \ref{L-C-domain},
its symmetry follows by inspection. Since $(L_{C},\mathcal{D})$ has a complete
system of eigenfunctions, it is essentially self-adjoint, i.e. it admits a
unique self-adjoint extension. The proof is finished.
\end{proof}

The modified ultrametric $d_{\ast}$ associated with the operator
$(L_{C},\mathcal{D})$ is defined by%
\begin{equation}
d_{\ast}(x,y)=\left\{
\begin{array}
[c]{ccc}%
1/\lambda(\{x\}\wedge\{y\}) & \text{if} & x\neq y\\
0 & \text{if} & x=y
\end{array}
\right.  . \label{d-modified}%
\end{equation}
Observe that the function $w: \mathcal{T}(X,d)\cup F_1\to [0,\infty)$,
$w(B):= 1/\lambda(B)$ and $w=0$ at each singleton, is a Whitney map.
By Proposition~\ref{WhitneyLemma},  $d_{\ast}$ is a proper ultrametric which induces the same topology and
the same collection of balls
as $d$ and, as one easily verifies,%
\[
\lambda(B)=\frac{1}{\diam_{\ast}(B)}.
\]

\begin{proof}[Proof of Theorem~\ref{generality of spectrum}]
Assume an ultrametric space $(X,d)$ and a set $S\subseteq\lbrack0,\infty)$ satisfy
the hypotheses of Theorem~\ref{generality of spectrum}. Let M be a countable dense
subset of $S^{-1}\cup
\{0\}$ containing $0$, where $S^{-1} = \{s^{-1}: s\in S, s>0 \}$. Let $d'$ be an
equivalent metric with $\R(d')= M$, as guaranteed by Theorem~\ref{ultrametric d-d'}
and  let $L_C$
be the hierarchical Laplacian on the ultrametric space $(X,d')$
corresponding to a choice function $C(B)$ such that  $\R(d') = \R(d'_{\ast})$. We can
 choose, for
instance, $C(B)=1/\diam'(B)-1/\diam'(B')$ where ball $B'$ is the immediate
predecessor of $B$.
Then we have
\[
 S = \overline{ M^{-1} } = \overline{\{ 1/d'(x,y); x \neq y \}} =
 \overline{\{ 1/d'_{\ast}(x,y); x \neq y \}} =
 \mathsf{Spec}(L_C)
\]
which completes the proof of Theorem~\ref{generality of spectrum}.
\end{proof}

\

Let $P_{t}=\exp(-tL_{C}),$ $t\geq0,$ be a symmetric contraction semigroup
generated by the self-adjoint operator $(L_{C},\mathsf{Dom}_{L_{C}})$.

\begin{proposition}
\label{Isotropic semigroup}The semigroup $\{P_{t}\}$ has the following
representation%
\[
P_{t}f(x)=\int\limits_{0}^{\infty}P_{B_{r}(x)}f\text{ }d\sigma
^{t}(r),\text{ }f\in L^{2},
\]
where $\sigma^{t}(r)=\exp(-t/r)$ and $B_{r}(x)$ is the $d_{\ast}$-ball of
radius $r$ centered at $x.$

In particular, $\{P_{t}\}$ is an isotropic Markov semigroup on the ultrametric
measure space $(X,d_{\ast},m)$ as defined and studied in \cite{BGPW}.
\end{proposition}

\begin{proof}
We choose $f=f_{B}$ and compute $P_{t}f(x).$ Using the identity
(\ref{f_series}) we obtain%
\begin{align*}
P_{t}f(x)  &  =\sum\limits_{T:\text{ }B\subseteq T}P_{t}f_{T,T^{\prime}%
}(x)=\sum\limits_{T:\text{ }B\subseteq T}e^{-t\lambda(T^{\prime}%
)}f_{T,T^{\prime}}(x)\\
&  =\sum\limits_{T:\text{ }B\subseteq T,x\in T^{\prime}}e^{-t\lambda
(T^{\prime})}\left(  f_{T}(x)-f_{T^{\prime}}(x)\right) \\
&  =\sum\limits_{T:\text{ }B\wedge\{x\}\subseteq T}e^{-t\lambda(T^{\prime}%
)}\left(  f_{T}(x)-f_{T^{\prime}}(x)\right)  -e^{-t\lambda(B\wedge
\{x\})}f_{B\wedge\{x\}}(x).
\end{align*}
Next observe that for any ball $T$ centered at $x,$%
\[
P_{T}\left(  f_{B}\right)  =\left\{
\begin{array}
[c]{ccc}%
f_{B}(x) & \text{if} & T\subseteq B\\
f_{T}(x) & \text{if} & B\subseteq T\\
0 & \text{if} & B\cap T=\varnothing
\end{array}
\right.  .
\]
With this observation in mind we write the equality from above as%
\[
P_{t}f(x)=\sum\limits_{T:\text{ }x\in T}e^{-t\lambda(T^{\prime})}\left(
P_{T}\left(  f_{B}\right)  (x)-P_{T^{\prime}}\left(  f_{B}\right)  (x)\right)
.
\]
Applying the Abel transformation and the definition (\ref{d-modified}) of the
modified ultrametric $d_{\ast}$ we get the desired equality with $f=f_{B}.$
The set spanned by the functions $f_{B}$ is dense in $L^{2},$ the result follows.
\end{proof}

\paragraph{ $L^{p}$-Spectrum of the hierarchical Laplacian}

Consider the semigroup $P_{t}=\exp(-tL_{C}).$ As $\{P_{t}\}$ is symmetric and
Markovian, it admits an extension to $L^{q},$ $1\leq q<\infty,$ as a
continuous contraction semigroup, call it $\{P_{t}^{q}\}$,%
\[
P_{t}^{q}f(x)=\int\limits_{0}^{\infty}P_{B_{r}(x)}f\text{ }d\sigma
^{t}(r),\text{ }f\in L^{q}.
\]
Let $(-\mathcal{L},\mathsf{Dom}_{\mathcal{L}})$ be the infinitesimal generator
of the semigroup $\{P_{t}^{q}\}$ . Since the operator $(L_{C},\mathcal{D)}$
acts in $L^{q}$ and $\{P_{t}^{q}\}$ extends $\{P_{t}\}$, the operator
$\mathcal{L}$ defines a closed extension of $L_{C}$, call it $L_{C}^{q}.$
Applying Theorem 7.8 in \cite{BGPW} we obtain

\begin{proposition}
\label{p-Spectrum} For any $1\leq q<\infty,$ the operator $(L_{C}%
,\mathcal{D)}$ acting in $L^{q}$ admits a closed extension $L_{C}^{q}.$ The
closed operator $-L_{C}^{q}$ coincides with the infinitesimal generator of a
Markov semigroup acting in $L^{q}$. Moreover, for all $1\leq q<\infty,$
\[
\mathsf{Spec}(L_{C}^{q})   =\mathsf{Spec}(L_{C}^{2})
  =\overline{\{\lambda(B):B\in\mathcal{B}\}}\cup\{0\}.
\]
\end{proposition}
\section{$p$-Adic Fractional Derivative}

\setcounter{equation}{0}Consider the field $\mathbb{Q}_{p}$ of $p$-adic
numbers endowed with the $p$-adic norm $\left\Vert x\right\Vert _{p}$ and the
$p$-adic ultrametric $d(x,y)=$ $\left\Vert x-y\right\Vert _{p}.$ Let $m$ be
the normalized Haar measure on $\mathbb{Q}_{p}$, that is, $m(\mathbb{Z}%
_{p})=1$, where $\mathbb{Z}_{p}$ is the set of $p$-adic integers.

In the ultrametric space $(\mathbb{Q}_{p},d)$ all $d$-balls are either compact
subgroups $p^{k}\mathbb{Z}_{p}$ or their cosets $p^{k}\mathbb{Z}_{p}+a$;
$\diam(p^{k}\mathbb{Z}_{p}+a)=p^{-k}$ and $m(p^{k}\mathbb{Z}_{p}+a)=p^{-k}$.
The balls form a regular tree $\mathcal{T}_{p}(X)$ of forward degree $p$
without the least element and without end-points.

The notion of $p$-adic fractional derivative related to the concept of
$p$-adic Quantum Mechanics was introduced in several mathematical papers
Vladimirov \cite{Vladimirov}, Vladimirov and Volovich
\cite{VladimirovVolovich}, Vladimirov, Volovich and Zelenov
\cite{Vladimirov94}. In particular, in \cite{Vladimirov} a one-parametric
family $\{(\mathfrak{D}^{\alpha},\mathcal{D})\}_{\alpha>0}$ of operators
(called operators of fractional derivative of order $\alpha$) have been
introduced. Recall that $\mathcal{D}$ is the set of all locally constant
functions having compact support.

The operators $\mathfrak{D}^{\alpha}$ were defined via Fourier transform
available on locally compact Abelian group $\mathbb{Q}_{p}$,
\begin{equation}
\widetilde{\mathfrak{D}^{\alpha}u}(\xi)=\left\Vert \xi\right\Vert _{p}%
^{\alpha}\widetilde{u}(\xi). \label{fracderiv}%
\end{equation}
Moreover, it was shown that each operator $\mathfrak{D}^{\alpha}$ can be
written as a Riemann-Liouville type singular integral operator%
\[
\mathfrak{D}^{\alpha}u(x)=\frac{p^{\alpha}-1}{1-p^{-\alpha-1}}\int
\limits_{\mathbb{Q}_{p}}\frac{u(x)-u(y)}{\left\Vert x-y\right\Vert
_{p}^{1+\alpha}}\text{ }dm(y).
\]
The aim of this section is to illustrate the results of Section 3 showing that
the operator $(\mathfrak{D}^{\alpha},\mathcal{D})$ is in fact a hierarchical
Laplacian. More precisely, we claim that $(\mathfrak{D}^{\alpha},\mathcal{D})$
is a hierarchical Laplacian corresponding to the choice function%
\begin{equation}
C(B)=(1-p^{-\alpha})\diam(B)^{-\alpha}, \label{C-alpha}%
\end{equation}
or equivalently, the eigenvalues $\lambda(B)$ are of the form%
\[
\lambda(B)=\diam(B)^{-\alpha}.
\]
In particular, the set $\mathsf{Spec}(\mathfrak{D}^{\alpha})$ consists of
eigenvalues $p^{k\alpha},k\in\mathbb{Z},$ each of which has infinite
multiplicity and contains $0$ as an accumulation point.

To prove the claim observe that the Fourier transform $\mathcal{F}%
:f\mapsto\widetilde{f}$ on the locally compact Abelian group $\mathbb{Q}_{p}$
is a linear isomorphism from $\mathcal{D}$ onto itself. This basic fact and
(\ref{fracderiv}) imply that $(\mathfrak{D}^{\alpha},\mathcal{D})$ is an
essentially self-adjoint and non-negative definite operator in $L^{2}%
=L^{2}(\mathbb{Q}_{p},m).$ Next we claim that the spectrum of the symmetric
operator $(\mathfrak{D}^{\alpha},\mathcal{D})$ coincides with the range of the
function $\xi\mapsto\left\Vert \xi\right\Vert _{p}^{\alpha},$ that is,
\[
\mathsf{Spec}(\mathfrak{D}^{\alpha})=\{p^{k\alpha}:k\in\mathbb{Z}\}\cup\{0\};
\]
the eigenspace $\mathcal{H(\lambda)}$ corresponding to the eigenvalue
$\lambda=p^{k\alpha},$ is spanned by the function
\[
f_{k}=\frac{1}{m(p^{k}\mathbb{Z}_{p})}\mathbf{1}_{p^{k}\mathbb{Z}_{p}}%
-\frac{1}{m(p^{k-1}\mathbb{Z}_{p})}\mathbf{1}_{p^{k-1}\mathbb{Z}_{p}}%
\]
and\ all its\ shifts\ $f_{k}(\cdot+a)$ with $a\in\mathbb{Q}_{p}/p^{k}%
\mathbb{Z}_{p}$.

Indeed, the ball $B_{s}(0)$, $p^{l}\leq s<p^{l+1}$, is the compact subgroup
$p^{-l}\mathbb{Z}_{p}$ of $\mathbb{Q}_{p}$, whence the measure $\omega
_{s}=\mathbf{1}_{B_{s}(0)}m/m(B_{s}(0))$ coincides with the normed Haar
measure of that compact
subgroup. Since for any locally compact Abelian group, the Fourier transform
of the normed Haar measure of any compact subgroup is the indicator of its
annihilator group and, in our particular case, the annihilator of the group
$p^{-l}\mathbb{Z}_{p}$ is the group $p^{l}\mathbb{Z}_{p}\,$, (see \cite{Feldman}), we obtain
\[
\widetilde{\omega_{s}}(\xi)=\mathbf{1}_{p^{l}\mathbb{Z}_{p}}(\xi
)=\mathbf{1}_{\{\left\Vert \xi\right\Vert _{p}\leq p^{-l}\}},\quad
\text{where\ \ }p^{l}\leq s<p^{l+1}.
\]
Computing now the Fourier transform of the function $f_{k},$%
\[
\widetilde{f_{k}}(\xi)=\mathbf{1}_{\{\left\Vert \xi\right\Vert _{p}\leq
p^{k}\}}-\mathbf{1}_{\{\left\Vert \xi\right\Vert _{p}\leq p^{k-1}%
\}}=\mathbf{1}_{\{\left\Vert \xi\right\Vert _{p}=p^{k}\}},
\]
we get%
\[
\widetilde{\mathfrak{D}^{\alpha}f_{k}}(\xi)=\left\Vert \xi\right\Vert
_{p}^{\alpha}\widetilde{f_{k}}(\xi)=p^{k\alpha}\widetilde{f_{k}}(\xi),
\]
as desired.

Finally, we apply Proposition \ref{Spectral properties} to conclude that the
essentially self-adjoint operator $(\mathfrak{D}^{\alpha},\mathcal{D})$
coincides with the hierarchical Laplacian $(L_{C},\mathcal{D)}$ with $C(B)$
given by the equation (\ref{C-alpha}).

At last, applying Proposition \ref{p-Spectrum} we obtain the following result.

\begin{proposition}
For any $1\leq q<\infty,$ the operator $\mathfrak{D}^{\alpha}$ admits a closed
extension $\mathfrak{D}_{q}^{\alpha}$ in $L^{q}.$ The closed operator
$-\mathfrak{D}_{q}^{\alpha}$ coincides with the infinitesimal generator of a
translation invariant Markov semigroup acting in $L^{q}.$ Moreover, for all
$1\leq q<\infty,$
\[
\mathsf{Spec}(\mathfrak{D}_{q}^{\alpha})=\mathsf{Spec}(\mathfrak{D}%
_{2}^{\alpha}).
\]
\end{proposition}

In the general setting of Propositions \ref{Spectral properties} and
\ref{p-Spectrum}, some eigenvalues may well have finite multiplicity and some
not. Indeed, attached to each ball $B$ of $d_{\ast}$-diameter $1/\lambda$
there are the eigenvalue $\lambda$ and the corresponding finite dimensional
eigenspace $\mathcal{H}_{B}.$ This eigenspace is spanned by the finitely many
functions
\[
f_{C,B}=\frac{1}{m(C)}\mathbf{1}_{C}-\frac{1}{m(B)}\mathbf{1}_{B}\,,
\]
where $C$ runs through all balls whose predecessor is $C^{\prime}=B$. Recall
that $\dim\mathcal{H}_{B}=l(B)-1$, where $l(B)=\sharp\{C\in\mathcal{B}%
:C^{\prime}=B\}$.

It follows that in general, if there exists only a finite number of distinct
balls of $d_{\ast}$-diameter $1/\lambda$ then the eigenvalue $\lambda$ has
finite multiplicity.

This is clearly not the case for the ultrametric measure space $(\mathbb{Q}%
_{p},d,m)\,$ and the operator $\mathfrak{D}^{\alpha}.$ Indeed, every $d_{\ast
}$-ball has its diameter in the set $\Lambda_{\alpha}=\{p^{k\alpha}%
:k\in\mathbb{Z\}}$, and each ball $B_{1/\lambda}(0)$ centered at the neutral
element $0$ and of diameter $1/\lambda$ has infinitely many disjoint
translates $\{a_{i}+B_{1/\lambda}(0)=B_{1/\lambda}(a_{i}),$ $i=1,2,...$, which cover
$\mathbb{Q}_{p}$ and are balls of the same diameter. Thus, all eigenvalues
have infinite multiplicity.

\begin{remark}
Let $\mathcal{H}(\lambda)$ be the eigenspace corresponding to the eigenvalue
$\lambda\in\Lambda_{\alpha}\,$. Then
\begin{equation}
L^{2}=\bigoplus_{\lambda\in{\Lambda}_{\alpha}}\mathcal{H}(\lambda
)\quad\text{and}\quad\mathcal{H}(\lambda)=\bigoplus_{i=1}^{\infty}%
\mathcal{H}_{a_{i}+B_{1/\lambda}(0)}. \label{L-2-decomposition}%
\end{equation}
We choose for each closed ball $B\subset\mathbb{Q}_{p}$ an orthonormal basis
$\{e_{i}^{B}:1\leq i\leq p-1\}$ in $\mathcal{H}_{B}.$ In view of
(\ref{L-2-decomposition}), the set of eigenfunctions $\{e_{i}^{B}%
:B\in\mathcal{B},1\leq i\leq p-1\}$ is an orthonormal basis in $L^{2}.$ (This
reasoning applies to arbitrary ultrametric spaces.) Whether this set is a
Schauder basis in $L^{q}$, $1\leq q<\infty,$ is an open question.
\end{remark}

\paragraph{Random perturbations}

Let $\mathfrak{D}$ be the operator of fractional derivative of order
$\alpha=1$ acting on the ultrametric measure space $(\mathbb{Q}_{p},d,m).$ For
simplicity we assume that $p=2.$ Let $\{\varepsilon(B):B\in\mathcal{B}\}$ be
i.i.d. Bernoulli random variables defined on a probability space $(\Omega,P).$
We define a perturbation $\mathfrak{D}(\omega),$ $\omega\in\Omega,$ of the
operator $\mathfrak{D}$ as follows%
\[
\mathfrak{D}(\omega)f(x)=-\sum\limits_{B\in\mathcal{B}:\text{ }x\in
B}C(B,\omega)\left(  P_{B}f-f(x)\right)  ,\text{ }f\in\mathcal{D},
\]
where the perturbated choice function $C(B,\omega)=C(B)(1+\delta
\varepsilon(B))$ with $C(B)$ given at (\ref{C-alpha}), $\alpha=1.$ Evidently
the operator $\mathfrak{D}(\omega)$ is a hierarchical Laplacian, for each
$\omega\in\Omega.$

\begin{proposition}\label{Spectrum at random}
In the notation introduced above,%
\[
\mathsf{Spec}(\mathfrak{D}(\omega))=\{0\}\cup\left\{  \cup_{l\in\mathbb{Z}%
}[2^{-l},2^{-l}(1+\delta)]\right\}  \text{ },\quad P-a.s.
\]
In particular, when $0<\delta<1,$ the set $\mathsf{Spec}(\mathfrak{D}%
(\omega))$ consists of disjoint intervals and $\{0\}$ whereas, when
$\delta\geq1,$ $\mathsf{Spec}(\mathfrak{D}(\omega))=\mathbb{R}_{+},$ $P$- a.s. .
\end{proposition}

\begin{proof}
Let $\mathcal{B}_{l}\subset\mathcal{B}$ be the family of all balls of diameter
$2^{l}.$ For $B\in\mathcal{B}_{l}$ and $\omega\in\Omega,$ let us compute the
eigenvalue $\lambda(B,\omega)$ of the operator $\mathfrak{D(\omega),}$%
\begin{align}
\lambda(B,\omega)  &  =\sum\limits_{T\in\mathcal{B}:\text{ }B\subseteq
T}C(T,\omega)=\lambda(B)+\delta\sum\limits_{T\in\mathcal{B}:\text{ }B\subseteq
T}\varepsilon(T,\omega)C(T)\label{lambda at random}\\
&  =\lambda(B)\left(  1+\delta\frac{C(B)}{\lambda(B)}\sum\limits_{T\in
\mathcal{B}:\text{ }B\subseteq T}\frac{C(T)}{C(B)}\varepsilon(T,\omega)\right)
\nonumber\\
&  =2^{-l}\left(  1+\delta U(B,\omega)\right)  ,\nonumber
\end{align}
where%
\[
U(B,\omega)=\sum\limits_{k\geq1}2^{-k}\varepsilon(B_{k},\omega)\quad\text{and
}\quad B=B_{1}\varsubsetneq B_{2}\varsubsetneq...\text{ .}%
\]
Notice that $U(B),$ $B\in\mathcal{B}_{l},$ are (dependent) identically
distributed random variables having values in the interval $[0,1].$

\emph{We claim that}%
\begin{equation}
P\{\omega:U(B,\omega)\in I\quad\text{for all }\quad B\in\mathcal{B}%
_{l}\}=0,\quad\label{claim}%
\end{equation}
$\emph{for}$ \emph{any dyadic interval} $I\varsubsetneq\lbrack0,1].$

Indeed, for any given $B\in\mathcal{B}_{l},$
\[
\{\omega:U(B,\omega)\in I\}=\{\omega:\varepsilon(B_{k},\omega)=\varepsilon
_{k}\quad\text{for all }\quad\text{ }k\leq\kappa\},
\]
for some $\varepsilon_{k}=\varepsilon_{k}(I)\in\{0,1\}$ and $\kappa
=\kappa(I).$ Let $\mathcal{B}_{l}^{\prime}\subset\mathcal{B}_{l}$ be an
infinite collection of balls such that, for each two balls $A$ and $B$ in
$\mathcal{B}_{l}^{\prime}$, the ball $A\wedge B$ belongs to $\mathcal{B}%
_{l+\kappa}.$ Since $\{\varepsilon(B):B\in\mathcal{B}\}$ are i.i.d., we obtain%
\begin{align*}
P\{\omega &  :U(B,\omega)\in I\text{ }\quad\text{for all }\quad B\in
\mathcal{B}_{l}\}\\
&  \leq P\{\omega:U(B,\omega)\in I\text{ }\quad\text{for all }\quad
B\in\mathcal{B}_{l}^{\prime}\}\\
&  =P\{\omega:\varepsilon(B_{k},\omega)=\varepsilon_{k}\text{ }\quad\text{for
all }\quad k\leq\kappa\text{ and }B\in\mathcal{B}_{l}^{\prime}\}\\
&  =\prod\limits_{B\in\mathcal{B}_{l}^{\prime}}P\{\omega:\varepsilon
(B_{k},\omega)=\varepsilon_{k}\text{ }\quad\text{for all }\quad k\leq
\kappa\}=0,
\end{align*}
as claimed.

At last, by (\ref{claim}), for any given dyadic interval $J\subset
\lbrack0,1],$ we have%
\begin{equation}
P\{\omega:U(B,\omega)\in J\text{ }\quad\text{for some }\quad B\in
\mathcal{B}_{l}\}=1.\label{P-U probability}%
\end{equation}
The equations (\ref{P-U probability}) and (\ref{lambda at random}) yield%
\[
\overline{\{\lambda(B,\omega):B\in\mathcal{B}_{l}\}}=[2^{-l},2^{-l}%
(1+\delta)]\quad P-a.s.,
\]
as desired. The proof is finished.
\end{proof}

\begin{figure}\flushright
\begin{tikzpicture}[
level 1/.style={sibling distance=2cm, level distance=1.5cm},
level 2/.style={sibling distance=4cm, level distance=3cm},
level 3/.style={sibling distance=1cm, level distance=0.75cm},
level 4/.style={sibling distance=1cm, level distance=.75cm},
level 5/.style={sibling distance=1cm, level distance=.75cm},
level 6/.style={sibling distance=.4cm, level distance=.3cm},
level 7/.style={sibling distance=1.6cm, level distance=1.2cm},
]
\node {$\infty$}
child [dash pattern=on 0pt off 8pt,line width=2pt,line cap=round] {
	node[draw,circle,solid,inner sep=1.5pt,radius=1pt,fill,line width=1pt,label={0:$B_L$},label={[label distance=6.5cm]0:$\mathcal{B}_L$}] (A1) { }
	child[solid, line width=1pt] {
		node[draw,circle,solid,inner sep=1.5pt,radius=1pt,fill,line width=1pt,label={0:$B_{L-1}$},label={[label distance=8.5cm]0:$\mathcal{B}_{L-1}$}] (B1) { }
		child[solid, line width=1pt] {
			node[draw,circle,solid,inner sep=1.5pt,radius=1pt,fill,line width=1pt] { }
			child [dash pattern=on 0pt off 8pt,line width=2pt,line cap=round] {
				node[draw,circle,solid,inner sep=1.5pt,radius=1pt,fill,line width=1pt] { }
				child [solid, line width=1pt] {
					node[draw,circle,solid,inner sep=1.5pt,radius=1pt,fill,line width=1pt,label={[label distance=10cm]0:$\mathcal{B}_{l}$}] (C1) { }
					child [solid, line width=1pt] {
						child [dash pattern=on 0pt off 8pt,line width=2pt,line cap=round] {
							node[draw,circle,solid,inner sep=1.5pt,radius=1pt,fill,line width=1pt,label={270:$0$},label={[label distance=11cm]0:$\mathbb{Q}_2$}] (D1) { }
						}
						child [opacity=0] {
							node (D2) { }
						}
					}
					child [solid, line width=1pt] {	}
				}
				child [solid, line width=1pt] {
					node[draw,circle,solid,inner sep=1.5pt,radius=1pt,fill,line width=1pt] { }
					child [solid, line width=1pt] {	}
					child [solid, line width=1pt] {	}
				}
			}
			child [opacity=0] { }
		}
		child[solid, line width=1pt] {
			node[draw,circle,solid,inner sep=1.5pt,radius=1pt,fill,line width=1pt] { }
			child [opacity=0] {	}
			child [dash pattern=on 0pt off 8pt,line width=2pt,line cap=round] {
				node[draw,circle,solid,inner sep=1.5pt,radius=1pt,fill,line width=1pt] { }
				child [solid, line width=1pt] {
					node[draw,circle,solid,inner sep=1.5pt,radius=1pt,fill,line width=1pt] { }
					child [solid, line width=1pt] {	}
					child [solid, line width=1pt] {	}
				}
				child [solid, line width=1pt] {
					node[draw,circle,solid,inner sep=1.5pt,radius=1pt,fill,line width=1pt] (C2) { }
					child [solid, line width=1pt] {	}
					child [solid, line width=1pt] {	}
				}
			}
		}
	}
	child[solid, line width=1pt] {
		node[draw,circle,solid,inner sep=1.5pt,radius=1pt,fill,line width=1pt] (B2) { }
		child[solid, line width=1pt] {
			node[draw,circle,solid,inner sep=1.5pt,radius=1pt,fill,line width=1pt] { }
			child [dash pattern=on 0pt off 8pt,line width=2pt,line cap=round] {
				node[draw,circle,solid,inner sep=1.5pt,radius=1pt,fill,line width=1pt] { }
				child [solid, line width=1pt] {
					node[draw,circle,solid,inner sep=1.5pt,radius=1pt,fill,line width=1pt] { }
					child [solid, line width=1pt] {	}
					child [solid, line width=1pt] {	}
				}
				child [solid, line width=1pt] {
					node[draw,circle,solid,inner sep=1.5pt,radius=1pt,fill,line width=1pt] { }
					child [solid, line width=1pt] {	}
					child [solid, line width=1pt] {	}
				}
			}
			child [opacity=0] {	}
		}
		child[solid, line width=1pt] {
			node[draw,circle,solid,inner sep=1.5pt,radius=1pt,fill,line width=1pt] { }
			child [opacity=0] {	}
			child [dash pattern=on 0pt off 8pt,line width=2pt,line cap=round] {
				node[draw,circle,solid,inner sep=1.5pt,radius=1pt,fill,line width=1pt] { }
				child [solid, line width=1pt] {
					node[draw,circle,solid,inner sep=1.5pt,radius=1pt,fill,line width=1pt] { }
					child [solid, line width=1pt] {	}
					child [solid, line width=1pt] {	}
				}
				child [solid, line width=1pt] {
					node[draw,circle,solid,inner sep=1.5pt,radius=1pt,fill,line width=1pt] { }
					child [solid, line width=1pt] {	}
					child [solid, line width=1pt] {	
						child [opacity=0] {}
						child [dash pattern=on 0pt off 8pt,line width=2pt,line cap=round] {	}
					}
				}
			}
		}
	}
}
child [opacity=0]{
	node[draw,circle,solid,inner sep=1.5pt,radius=1pt,fill,line width=1pt,opacity=0] (A2) { }
};
\draw[shorten >=-5cm,,shorten <=2cm, dash pattern=on 0cm off 1cm,line width=2pt,line cap=round] (A1)--(A2);
\draw[shorten >=3cm,,shorten <=-6cm, dash pattern=on 0cm off 1cm,line width=2pt,line cap=round] (A1)--(A2);
\draw[shorten >=-5cm,,shorten <=6cm, dash pattern=on 0cm off 1cm,line width=2pt,line cap=round] (B1)--(B2);
\draw[shorten >=5cm,,shorten <=-4cm, dash pattern=on 0cm off 1cm,line width=2pt,line cap=round] (B1)--(B2);
\draw[shorten >=-7.5cm,,shorten <=8.5cm, dash pattern=on 0cm off 1cm,line width=2pt,line cap=round] (C1)--(C2);
\draw[shorten >=3.5cm,,shorten <=-2.5cm, dash pattern=on 0cm off 1cm,line width=2pt,line cap=round] (C1)--(C2);

\draw[shorten >=-9.5cm,,shorten <=-1.5cm, dash pattern=on .1cm off .4cm, line width=1.5pt,line cap=round] (D1)--(D2);
\end{tikzpicture}
\caption{Defining the arithmetic mean eigenvalue $\overline{\lambda_{l}}(B_{L},\omega)$}
\end{figure}
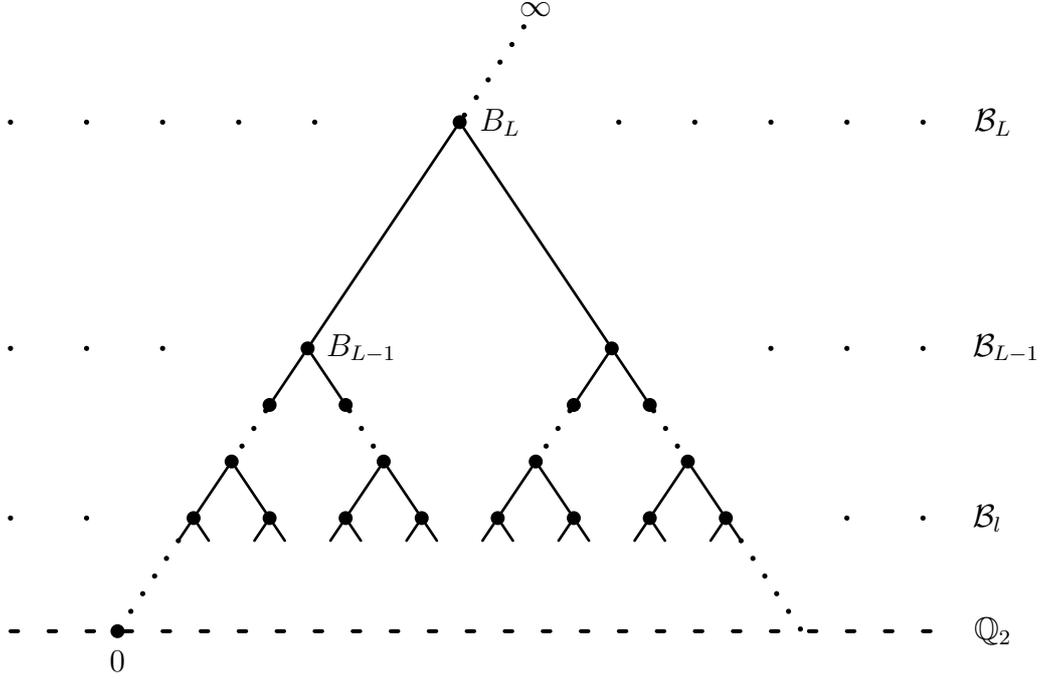

Let, as in the proof of Proposition \ref{Spectrum at random},%
\[
U(B,\omega)=\sum\limits_{k\geq1}2^{-k}\varepsilon(B_{k},\omega)\quad\text{and
}\quad B=B_{1}\varsubsetneq B_{2}\varsubsetneq...\text{ .}%
\]
Choose a reference point $o$ and let $B_{L}$ $\in\mathcal{B}_{L}$ be the ball
centred at $o.$ Fix a level $\mathcal{B}_{l}$ and define the arithmetic mean
$\overline{U}(B_{L},\omega),$ $L>l,$ as%
\[
\overline{U}(B_{L},\omega)=\frac{1}{2^{L-l}}\sum\limits_{B:\text{ }B\subset
B_{L},B\in\mathcal{B}_{l}}U(B,\omega).
\]
Let $\mathsf{E} \left[ \varepsilon(B)\right] =p$ and
$\mathsf{Var}\left[\varepsilon(B)\right]=pq,$
$p+q=1.$ Using the tree-structure of the ultrametric measure space
$(\mathbb{Q}_{2},d,m)$ and the fact that $\{\varepsilon(B):B\in\mathcal{B}\}$
are i.i.d. Bernoulli random variables we compute

\begin{enumerate}
\item  $\mathsf{E}\left[  \overline{U}(B_{L})\right]  =p,$

\item $\mathsf{Var}\left[  \overline{U}(B_{L})\right]  $ $\sim pq/2^{L-l},$ as
$L\rightarrow\infty.$
\end{enumerate}

In particular, $\sum\limits_{L\geq l}\mathsf{Var}\left[  \overline{U}%
(B_{L})\right]  <\infty.$ It follows that, as $L\rightarrow\infty,$%
\begin{equation}
\overline{U}(B_{L},\omega)\longrightarrow p\text{ }\quad P-a.s.\label{LLN}%
\end{equation}
and%
\begin{equation}
\frac{\overline{U}(B_{L},\omega)-p}{\sqrt{2^{-L+l}pq}}\longrightarrow
N(0,1)\text{  }\quad\text{in law},\label{CLT}%
\end{equation}
where $N(0,1)$ is the standard normal random variable.

For a given level $\mathcal{B}_{l}$ we define the arithmetic mean eigenvalue
$\overline{\lambda_{l}}(B_{L},\omega),$ $L>l,$ as
\[
\overline{\lambda_{l}}(B_{L},\omega)=\frac{1}{2^{L-l}}\sum\limits_{B:\text{
}B\subset B_{L},B\in\mathcal{B}_{l}}\lambda(B,\omega).
\]

Applying (\ref{LLN}), (\ref{CLT}) and the equation (\ref{lambda at random}) we
obtain the following result.

\begin{proposition}
\label{Limit theorem}In the notation introduced above, \[
\overline{\lambda_{l}}(B_{L},\omega)\longrightarrow2^{-l}(1+\delta p)\quad
P-a.s.,
\]
and%
\[
\frac{\overline{\lambda_{l}}(B_{L},\omega)-2^{-l}(1+\delta p)}{\delta
\sqrt{2^{-L-l}pq}}\longrightarrow N(0,1)\quad\text{in law},
\]
for any given level $\mathcal{B}_{l}$ $,$ and as $L\rightarrow\infty.$
\end{proposition}

%\bibliography{ref}

\begin{thebibliography}{99}                                                                                               %


\bibitem {BGP}A. Bendikov, A. Grigor'yan and Ch. Pittet, \textit{On a class of
Markov semigroups on discrete ultra-metric spaces}, Potential Anal.
\textbf{37} (2012), 125--169.

\bibitem {BGPW}A. Bendikov, A. Grigor'yan, Ch. Pittet and W. Woess,
\textit{Isotropic Markov semigroups on ultra-metric spaces}, arXiv:1304.6271

\bibitem {B}N. Bourbaki, \textit{Elements of Mathematics. General Topology,
Part II}, Hermann, Addison-Wesley, 1966.

\bibitem {Bovier}A. Bovier, \textit{The density of states in the Anderson
model at weak disorder: a renormalization group analysis of the hierarchical
model}, J. Statist. Phys. \textbf{59 } (1990), no. 3-4, 745-779.


\bibitem {Dyson1}F.J. Dyson, \textit{The dynamics of a disordered linear
chain}, Phys. Rev. \textbf{92} (1953), 1331-1338.

\bibitem {Dyson2}F.J. Dyson, \textit{Existence of a phase-transition in a
one-dimensional Ising ferromagnet}, Comm. Math. Phys. \textbf{12} (1969), 91-107.

\bibitem {Eng}R. Engelking, \textit{Theory of Dimensions, Finite and
Infinite}, Heldermann, 1995.

\bibitem{Feldman} G. M. Feldman, \textit{Functional Equations and Characterization
Problems on Locally Compact Abelian Groups}, Tracts in Mathematics \textbf{5},
European Mathematical Society, 2008.

\bibitem {deG}J. de Groot, \textit{Non-Archimedean metrics in topology}, Proc.
Amer. Math. Soc. \textbf{7} (1956), 948--956.

\bibitem {IN}A. Illanes and S. B. Nadler, Jr., \textit{Hyperspaces:
Fundamentals and Recent Advances}, Marcell Dekker, 1999.


\bibitem {Kech}A. Kechris, \textit{Classical descriptive set theory},
Springer, 1995.

\bibitem {Kritch1}E. Kritchevski, \textit{Spectral localization in the
hierarchical Anderson model}, Proc. Amer. Math. Soc. \textbf{135}, no. 5, 1431-1440.

\bibitem {Kritch2}E. Kritchevski, \textit{Hierarchical Anderson model}, Centre
de Recherches Math. CRM Proc. and Lecture Notes, vol. 42, 2007.

\bibitem {Kritch3}E. Kritchevski, \textit{Poisson Statistics of Eigenvalues in
the Hierarchical Anderson Model}, Ann. Henri Poincare \textbf{9} (2008), 685-709.

\bibitem {Lem}A. Lemin and V. Lemin, \textit{On uniform rationalization of
ultrametrics}, Topology Proc. \textbf{22} (1997), 275--283.

\bibitem {Lu}J. Luukkainen and H. Movahedi-Lankarani, \textit{Minimal
bi-Lipschitz embedding dimension of ultrametric spaces}, Fund. Math.
\textbf{144} (1994), 181--193.

\bibitem {Molchanov}S.A. Molchanov, \textit{Hierarchical random matrices and
operators. Application to Anderson model}, Proc. of 6th Lukacs Symposium
(1996), 179-194.

\bibitem {V}H. E. Vaughan, \textit{On locally compact metrisable spaces},
Bull. Amer. Math. Soc. \textbf{43} (1937), 532--535.

\bibitem {Vladimirov}V.S. Vladimirov, \textit{Generalized functions over the field of
$p$-adic numbers}, Uspekhi Mat. Nauk \textbf{43 }(1988), 17-53, 239.

\bibitem {VladimirovVolovich}V.S. Vladimirov and I.V. Volovich, \textit{$p$-adic
Schr\"{o}dinger-type equation}, Lett. Math. Phys. \textbf{18 }(1989), 43-53.

\bibitem {Vladimirov94}V.S. Vladimirov, I.V. Volovich and E.I. Zelenov,
\textit{$p$-adic analysis and mathematical physics}, Series on Soviet and East European
Mathematics, vol.1, World Scientific Publishing Co., Inc., River Edge, NY, 1994.
\end{thebibliography}

\end{document}